\documentclass{article}
\usepackage{amsmath,amssymb}
\usepackage{amscd}
\usepackage{comment}
\usepackage{color}
\usepackage{enumerate}
\usepackage{cases}
\usepackage{mathtools}
\usepackage[all]{xy}
\usepackage{amsthm}
\usepackage{pb-diagram}
\theoremstyle{definition}
\newtheorem{defi}{Definition}
\newtheorem{theo}[defi]{Theorem}
\newtheorem{lemm}[defi]{Lemma}
\newtheorem{prop}[defi]{Proposition}

\newtheorem{rem}[defi]{Remark}

\def\det{{\rm det}}

\def\SL{{\rm SL}}
\def\SU{{\rm SU}}

\def\End{{\rm End}}
\def\Exp{{\rm Exp}}

\def\Exp{{\rm Exp}}

\def\det{{\rm det}}

\def\Tr{{\rm Tr}}
\def\diag{{\rm diag}}
\def\Herm{{\rm Herm}}
\def\deg{{\rm deg}}

\def\R{{\mathbb R}}
\def\Z{{\mathbb Z}}
\def\C{{\mathbb C}}

\def\Q{{\mathbb Q}}

\def\g{\mathfrak{g}}
\def\t{\mathfrak{t}}
\def\h{\mathfrak{h}}

\def\inum{{\sqrt{-1}}}

\begin{document}
\title {On diagonal pluriharmonic metrics of $G$-Higgs bundles}
\author {Natsuo Miyatake}
\date{}
\maketitle
\begin{abstract} Let $(E,\Phi)\rightarrow (X,\omega_X)$ be a Higgs bundle over a compact K\"ahler manifold. We suppose that the holomorphic vector bundle $E$ decomposes into a direct sum of holomorphic line bundles. In this paper, we give the necessary and sufficient condition for the existence of a diagonal metric which is a solution to the Hermitian-Einstein equation. Our theorem can easily be generalized to $G$-Higgs bundles. We also describe the relationship between the stability condition and our condition using the torus action on the space of Higgs fields.
\end{abstract}

\section{Main theorem and proof}
Let $X$ be a compact connected Riemann surface of genus $g(X)\geq 2$ with the canonical bundle $K_X\rightarrow X$. We define a holomorphic vector bundle $E$ as the direct sum of holomorphic line bundles $L_1,\dots,L_r$ over $X$ with $\deg(L_1)+\cdots+\deg(L_r)=0$. Let $\Phi_j $ be a holomorphic section of $(L_j^{-1}L_{j+1})\otimes K_X$ for each $\ j=1,\dots, r-1$, and $\Phi_r$ a holomorphic section of $(L_r^{-1}L_1)\otimes K_X$. We define $\Phi$ as the sum of all the $\Phi_j$, $j=1,\dots,r$, and it belongs to the space $H^0(\End E\otimes K_X)$.
The pair $(E,\Phi)$ is called a cyclic Higgs bundle \cite{Bar1, DL1}. Suppose that $\Phi_j\neq 0$ for all $j=1,\dots, r$. Then $(E,\Phi)$ is stable, and thus, there uniquely exists a harmonic metric $h$ on $(E,\Phi)$. The harmonic metric $h$ splits as $h=(h_1,\dots, h_r)$ concerning to the decomposition since the Higgs field $\Phi$ is an eigensection of the gauge transformation $g=\diag(1,\omega,\dots, \omega^{r-1}) \ (\omega=e^{2\pi\inum/r})$ \cite{Bar1}. 
This property of the harmonic metric makes it possible to use techniques of maximum principles when we investigate it (see \cite{DL1}). The present paper aims to establish the following two results regarding Higgs bundles over compact K\"ahler manifolds: First, we provide a necessary and sufficient condition for the existence of a diagonal metric that solves the Hermitian-Einstein equation (see Theorem \ref{main theorem 1}). Second, we establish the relationship between the condition which we give in Theorem \ref{main theorem 1}  and the usual stability condition of Higgs bundles (see Proposition \ref{main theorem 3} and Proposition \ref{main theorem 4}). The precise statement of Theorem \ref{main theorem 1} is as follows: Let $(E,\Phi)\rightarrow (X,\omega_X)$ be a Higgs bundle over a compact K\"ahler manifold. Suppose that the holomorphic vector bundle $E$ decomposes as $E=L_1\oplus\cdots\oplus L_r$ with holomorphic line bundles $L_1,\dots, L_r\rightarrow X$. Then the Higgs field $\Phi$ decomposes as $\Phi=\Phi_0+\sum_{i,j=1,\dots, r}\Phi_{i,j}$, where $\Phi_{i,j}$ is a holomorphic $(1,0)$-form with values in $L_j^{-1}{L_i}$ and $\Phi_0$ is the diagonal part. For each $j=1,\dots, r$, let $\gamma_j$ be the degree of the holomorphic line bundle $L_j$ with respect to the K\"ahler form $\omega_X$. We assume that $\deg_{\omega_X} (E)=\gamma_1+\cdots+\gamma_r=0$ for simplicity.  Let $V$ be a real vector space defined as $V\coloneqq \{(x_1,\dots, x_r)\in\R^r\mid x_1+\cdots +x_r=0\}$. For each $i,j=1,\dots, r$, let $v_{i,j}\in V$ be a vector defined as $v_{i,j}\coloneqq u_i-u_j$, where $u_1,\dots, u_r$ is the canonical basis of $\R^r$. We define a vector $\gamma\in V$ as $\gamma\coloneqq (\gamma_1,\dots, \gamma_r)$. Then the following holds:
\begin{theo}\label{main theorem 1} {\it The Higgs bundle $(E,\Phi)$ admits a solution to the Hermitian-Einstein equation which is diagonal concerning the decomposition $E=L_1\oplus\cdots\oplus L_r$ if and only if the following (i) and (ii) hold:
\begin{enumerate}[(i)]
\item The off-diagonal part of $\Lambda[\Phi\wedge\Phi^{\ast h}]$ vanishes for a diagonal metric $h=(h_1,\dots, h_r)$, where $\Lambda$ denotes the adjoint of $\omega_X\wedge$. \label{off-diagonal vanishes}
\item The following holds:
\begin{align}
-\gamma\in\sum_{i,j=1,\dots, r, \Phi_{i,j}\neq0}\R_{>0}v_{i,j}.\label{gamma}
\end{align}
\end{enumerate}
In particular, a Higgs bundle satisfying (i) and (ii) is polystable.}
\end{theo}
\begin{rem} From \cite[Proposition 3.4]{Sim1}, if the above condition (i) and (ii) hold and in addition, $c_2(E)=0$, then the solution of the Hermitian-Einstein equation is a diagonal pluriharmonic metric. The corresponding pluriharmonic map from the universal covering space $\widetilde{X}$ to $\SL(r,\C)/\SU(r)$ naturally lifts to the map to $\SL(r,\C)/T$, where $T$ is the maximal torus of $\SU(r)$ which consists of all diagonal matrices of $\SU(r)$ (see \cite{Bar1, Bar2}).
\end{rem}
\begin{rem} Theorem \ref{main theorem 1} can easily be generalized to $G$-Higgs bundles. In particular, the following holds: let $G$ be a complex simple Lie group with Lie algebra $\g$. Let $H\subseteq G$ be a maximal torus with Lie algebra $\h$. Let $\alpha_1,\dots, \alpha_l$ be a base of the root space $\Delta\subseteq \h^\ast$ and $\delta$ the highest root. Let $\g=\h\oplus\bigoplus_{\alpha\in\Delta}\g_\alpha$ be the root space decomposition. Consider a $G$-Higgs bundle $(P_G,\Phi)\rightarrow (X,\omega_X)$ with holomorphic $G$-bundle $P_G$ which admits a reduction $P_H\subseteq P_G$ to a holomorphic $H$-subbundle $P_H$. Suppose that the Higgs field $\Phi$ is of the following form:
\begin{align*}
\Phi=\sum_{\alpha\in\{-\alpha_1,\dots, -\alpha_l,\delta\}}\Phi_\alpha.
\end{align*}
We call $(P_G,\Phi)$ a $G$-cyclic Higgs bundle over $(X,\omega_X)$. For a $G$-cyclic Higgs bundle, 
if $\Phi_{\alpha}\neq0$ for all $\alpha\in\{-\alpha_1,\dots, -\alpha_l, \delta\}$, then there uniquely exists a diagonal solution to the Hermitian-Einstein equation of the $G$-cyclic Higgs bundle $(P_G,\Phi)$. This can be verified by the following observation: For each $\alpha\in\{-\alpha_1,\dots, -\alpha_l,\delta\}$, let $h_\alpha$ be the coroot of $\alpha$. Then we can observe that
\begin{align*}
\sum_{\alpha\in\{-\alpha_1,\dots, -\alpha_l,\delta\}}\R_{>0}h_\alpha=\inum\t,
\end{align*}
where $\t$ is the Lie algebra of the maximal compact torus $T\subseteq H$. Therefore, no matter which direction the vector $\gamma$ points in, we see that condition (\ref{gamma}) holds for the $G$-cyclic Higgs bundle. This implies the claim (see also \cite[Section 2]{Miy1}).
\end{rem}
\begin{rem} We prove Theorem \ref{main theorem 1} by applying \cite[Theorem 1]{Miy1}. However, Theorem \ref{main theorem 1} can also be derived by showing the stability of quiver bundles \cite{AG1} coincides with condition (\ref{gamma}) in the setting we are considering. We also note that the same condition as (\ref{gamma}) is obtained in \cite{Bap1} for a slightly different PDE.
\end{rem}
We prove Theorem \ref{main theorem 1}. We can easily check that condition (\ref{off-diagonal vanishes}) is a necessary condition for the existence of a diagonal metric which solves the Hermitian-Einstein equation. Suppose that (\ref{off-diagonal vanishes}) holds. We fix a diagonal metric $h=(h_1,\dots h_r)$ such that $\det(h)=1$. Let $\xi\coloneqq (f_1,\dots, f_r)\ (f_1+\cdots+f_r=0)$ be a pair of $\R$-valued functions. Then the Hermitian-Einstein equation for  a metric $(e^{f_1}h_1,\dots, e^{f_r}h_r)$ is the following:
\begin{align}
\Delta_{\omega_X} \xi+\sum_{i,j=1,\dots, r}4|\Phi_{i,j}|_{h, \omega_X}^2e^{(v_{i,j},\xi)}v_{i,j}=-2\inum \Lambda F_h, \label{KW}
\end{align}
where we denote by $F_h$ the curvature of the Chern connection of $h$ and by $\Delta_{\omega_X}$ the geometric Laplacian. We apply \cite[Theorem 1]{Miy1} to equation (\ref{KW}).
In order to apply \cite[Theorem 1]{Miy1}, we must check that for each $i,j=1,\dots, r$, if $\Phi_{i,j}$ is not a zero section, then $\log|\Phi_{i,j}|_{h, \omega_X}^2$ is integrable. This follows from the following lemma, which is an immediate consequence of the fact that plurisubharmonic functions are locally integrable:
\begin{lemm}\label{L1loc}{\it Let $U\subseteq \C^n$ be a domain and $f:U\rightarrow V$ a holomorphic section of a trivial bundle $V=U\times\C^r\rightarrow U$. Then for any smooth Hermitian metric $h_V$ on $V$, $\log|f|_{h_V}\in L^1_{loc}(U)$.
}
\end{lemm}
\begin{proof} Let $v_1,\dots, v_r:U\rightarrow V$ be a holomorphic frame of $V$. Then $f$ is denoted as $f=f_1v_1+\cdots+f_rv_r$. We denote by $\hat{h}_V:U\rightarrow \Herm(r)$ the Hermitian matrix valued smooth function whose $(i,j)$ component is $h_V(v_i,v_j)$. Then $\hat{h}_V$ is diagonalized as
$^t\bar{B}\hat{h}_VB=\diag(\lambda_1,\dots,\lambda_r)$
for a unitary matrix valued function $B$ and positive functions $\lambda_1,\dots, \lambda_r$ over $U$, where we denote by $\diag(\lambda_1,\dots,\lambda_r)$ the diagonal matrix whose diagonal entries are $\lambda_1,\dots, \lambda_r$. We set
$(f_1^\prime,\dots, f_r^\prime)\coloneqq (f_1,\dots, f_r)B.$
Let $F\subseteq U$ be a compact subset of $U$. We define a positive constant $C$ as $C\coloneqq \min_{1\leq i\leq r}\{\min_{x\in F}\lambda_1(x),\dots, \min_{x\in F}\lambda_r(x)\}$. Then we have
\begin{align*}
\log|f|_{h_V}^2&=\log\{\lambda_1|f^\prime_1|^2+\cdots+\lambda_r|f^\prime_r|^2\} \\
&\geq \log\{C|f_1^\prime|^2+\cdots+C|f_r^\prime|^2\} \\
&=\log C+\log\{|f_1^\prime|^2+\cdots+|f_r^\prime|^2\} \\
&=\log C+\log\{|f_1|^2+\dots+|f_r|^2\}
\end{align*}
for each point of $F$. Since $f_1,\dots, f_r$ are holomorphic functions, $\log\{|f_1|^2+\dots+|f_r|^2\}$ is a plurisubharmonic function (see \cite{Dem1}). In particular, it is in $L^1_{loc}(U)$. This implies the claim.
\end{proof}
Then we prove Theorem \ref{main theorem 1}:
\begin{proof}[Proof of Theorem \ref{main theorem 1}] As already remarked, under the assumption of (\ref{off-diagonal vanishes}), the Hermitian-Einstein equation for a diagonal metric $h=(e^{f_1}h_1,\dots, e^{f_r}h_r)$ is equation (\ref{KW}). Then from \cite[Theorem 1]{Miy1}, equation (\ref{KW}) has a smooth solution if and only if (ii) holds. This implies the claim.
\end{proof}

\section{Relationship between condition (\ref{gamma}) and the stability condition}
We describe the relationship between condition (\ref{gamma}) and the usual stability condition of Higgs bundles by using the torus action on the space of Higgs fields. Suppose that $\gamma_1,\dots, \gamma_r$ are all rational numbers.  For each subset $I\subseteq\{1,\dots, r\}$, we define a subundle $E_I\subseteq E$ as $E_I\coloneqq \bigoplus_{i\in I}L_i$. Then the following holds:
\begin{prop}\label{main theorem 3}{\it The following are equivalent:
\begin{enumerate}[(i)] 
\item The following holds:
\begin{align}
-\gamma\in\sum_{\substack{i,j=1,\dots, r, \\ \Phi_{i,j}\neq 0}}\Q_{\geq 0}v_{i,j}. \label{Q1}
\end{align}
\item For any subset $I\in\{1,\dots, r\}$, if $E_I$ is preserved by $\Phi$, then $\deg_{\omega_X}(E_I)\leq 0$.
\end{enumerate}
In particular, if $(E,\Phi)$ is semistable, then $(E,\Phi)$ satisfies (\ref{Q1}) for any holomorphic splitting. 
}
\end{prop}

\begin{prop}\label{main theorem 4}{\it The following (i) and (ii) are equivalent:
\begin{enumerate}[(i)]
\item The following holds:
\begin{align}
-\gamma\in\sum_{\substack{i,j=1,\dots, r, \\ \Phi_{i,j}\neq 0}}\Q_{>0}v_{i,j}. \label{Q2}
\end{align}
\item There exist subsets $I_1,\dots, I_k\subseteq\{1,\dots, r\}$ such that 
\begin{itemize}
\item The holomorphic vector bundle $E$ is a direct sum of $E_{I_1},\dots, E_{I_k}$: $E=E_{I_1}\oplus\cdots\oplus E_{I_k}$ and for each $j=1,\dots, k$, $E_{I_j}$ is preserved by $\Phi$.
\item For each $j=1,\dots, k$, $\deg_{\omega_X}(E_{I_j})=0$.
\item For each $j=1,\dots, k$, if there exists a subset $I\subsetneq I_j$ and if $E_I$ is preserved by $\Phi$, then $\deg_{\omega_X}(E_I)<0$.
\end{itemize}
\end{enumerate}
In particular, if $(E,\Phi)$ is stable, then $(E,\Phi)$ satisfies (\ref{Q2}) for any holomorphic splitting.
}
\end{prop}
\begin{rem} The choice of the holomorphic splitting of $E$ is of course not unique. For example, suppose that the rank of $E$ is $2$, $E$ decomposes as $E=L_1\oplus L_2$, and there exists a non-trivial holomorphic bundle map $f:L_1\rightarrow L_2$. Then $E=L_1^\prime\oplus L_2$ is another decomposition, where $L_1^\prime$ denotes a holomorphic line bundle defined as $L_1^\prime\coloneqq \{(v_1,v_2)\in L_1\oplus L_2\mid v_2=f(v_1)\}$. 
\end{rem}
Before starting the proof of Proposition \ref{main theorem 3} and Proposition \ref{main theorem 4}, we make some preparations. Let $T$ be the maximal torus of $\SU(r)$ which consists of all diagonal matrices of $\SU(r)$. We denote by $\t$ the Lie algebra of $T$. Let $H\subseteq \SL(r,\C)$ be the complexification of $T$ with Lie algebra $\h=\t\oplus\inum\t$. We define a lattice $\h_\Z$ of $\inum\t$ as $\h_\Z\coloneqq \{\diag(n_1,\dots, n_r)\mid n_1,\dots,n_r\in\Z, n_1+\cdots +n_r=0\}$. Note that $\h_\Z$ coincides with the kernel of the exponential map $\Exp:\h\rightarrow H$. We regard the vector $\gamma$ and the vectors $v_{i,j} \ (i,j=1,\dots, r)$ as elements of $\h_\Q\coloneqq \h_\Z\otimes_\Z\Q$ by identifying the canonical basis $u_1,\dots, u_r$ with $\diag(1,\dots, 0), \dots, \diag(0,\dots, 1)$. Let $(\cdot,\cdot)$ be a bilinear form on $\h$ defined as $(u,v)\coloneqq \Tr(uv)$ for $u,v\in\h$. We denote by $\gamma^\ast\in\h_\Q^\ast$ be the dual of the vector $\gamma$ with respect to the $(\cdot, \cdot)$. We take a positive integer $n$ so that $n\gamma^\ast\in\h_\Z^\ast$. Then we define an element $\gamma^\vee\in\h_\Z^\ast$ as $\gamma^\vee\coloneqq -n\gamma^\ast$. Also, we define a character $\chi_{\gamma^\vee}:H\rightarrow \C^\ast$ as $\chi_{\gamma^\vee}(\Exp(v))\coloneqq e^{\langle v, \gamma^\vee\rangle}$ for $v\in\h$, where we denote by $\langle\cdot,\cdot\rangle$ the coupling between $\h$ and $\h^\ast$. We define an action of the algebraic torus $H$ on $H^0(\End E\otimes \bigwedge^{1,0})$ as 
$g\cdot(\theta,z)\coloneqq (\theta_0+\sum_{i,j=1,\dots,r}g_j^{-1}g_i\theta_{i,j}, \chi_{\gamma^\vee}^{-1}(g)z)$ for $(\theta,z)\in H^0(\End E\otimes \bigwedge^{1,0})\times\C$, where $g=\diag(g_1,\dots, g_r)$ is an element of $H$ and we denote by $\theta_0$ and $\theta_{i,j} \ (i,j=1,\dots, r)$ the diagonal component and the $(i,j)$-component of $\theta$, respectively. For each $i,j=1,\dots, r$, let $\alpha_{i,j}$ be the element of $\h^\ast$ defined as $\alpha_{i,j}\coloneqq \lambda_i-\lambda_j$, where $\lambda_1,\dots, \lambda_r\in\h^\ast$ is the dual basis of $u_1,\dots, u_r$. For each $s\in\inum\t$, we define a set $A_s$ as $A_s\coloneqq \{(i,j)\mid i,j\in\{1,\dots, r\}, \alpha_{i,j}(s)\geq 0\}$. Then we start the proof of Proposition \ref{main theorem 3} and Proposition \ref{main theorem 4}.
\begin{proof}[Proof of Proposition \ref{main theorem 3}]
Condition (\ref{Q1}) holds if and only if for each $z\neq 0$, the closure of $H\cdot (\Phi,z)$ does not intersects $\C\times\{0\}$ (see \cite[Proposition 3]{Miy1}). Furthermore, this is equivalent to the following:
\begin{itemize}
\item It holds that $\gamma^\vee(s)\geq 0$ for any $s\in\inum\t$ such that 
\begin{align}
\Phi\in H^0(\bigoplus_{(i,j)\in A_s}L_j^{-1}L_i\otimes\bigwedge^{1,0}). \label{Phi}
\end{align}
\end{itemize}

Let $s=\diag(s_1,\dots, s_r)$ be an element of $\inum\t$ such that (\ref{Phi}) holds. We may assume that $s_1\geq\cdots\geq s_r$. Let $I_j\subseteq\{1,\dots,r\}$ be a subset defined as $I_j\coloneqq \{1,\dots, j\}$.
From (\ref{Phi}), the Higgs field $\Phi$ preserves the sequence of subbundles 
$0\subsetneq E_{I_{j_1}}\subsetneq E_{I_{j_2}}\subseteq \cdots\subsetneq E_{I_{j_k}}\subsetneq E$,
where $1\leq j_1<\cdots <j_k< r$ are all the elements of $J_s\coloneqq \{j\in\{1,\dots, r-1\}\mid s_j>s_{j+1}\}$. Then $\gamma^\vee(s)$ can be calculated as
\begin{align}
\gamma^\vee(s)=&-n\Tr(\gamma s) \notag \\
=&-n(\gamma_1s_1+\cdots +\gamma_rs_r) \notag \\
=&-n(\gamma_1(s_1-s_2)+(\gamma_1+\gamma_2)(s_2-s_3)+\cdots \notag \\
&+(\gamma_1+\cdots+\gamma_{r-1})(s_{r-1}-s_r)) \notag \\
=&-n\sum_{j\in J_s}(s_j-s_{j+1})\deg_{\omega_X}(E_{I_j}). \label{E}
\end{align}
Therefore, it holds that $\gamma^\vee(s)\geq 0$ if and only if $\deg_{\omega_X}(E_{I_j})\leq 0$ for all $j\in J_s$. This implies the claim. 
\end{proof}
\begin{proof}[Proof of Proposition \ref{main theorem 4}]
Condition (\ref{Q2}) holds if and only if $H\cdot (\Phi,z)$ is closed for each $z\neq 0$ (see \cite[Proposition 4]{Miy1}). Furthermore, this is equivalent to that the following holds for all $s\in\inum\t$ such that (\ref{Phi}) holds:
\begin{itemize}
\item The element $s$ of $\inum\t$ satisfies $\gamma^\vee(s)\geq 0$ and if $\gamma^\vee(s)=0$, then $\Phi$ lies in 
$H^0(\bigoplus_{(i,j)\in A_s^0}L_j^{-1}L_i\otimes\bigwedge^{1,0})$, 
where $A_s^0$ is defined as $A_s^0\coloneqq \{(i,j)\mid i,j\in\{1,\dots, r\}, \alpha_{i,j}(s)=0\}$.
\end{itemize}
Then from (\ref{E}), we see that (i) and (ii) are equivalent.
\end{proof}
\begin{rem} For the proof of Proposition \ref{main theorem 3} and Proposition \ref{main theorem 4}, the author referred the definition of the stability of Higgs bundles by using parabolic subalgebras (see, for example, \cite{GGM1}). 
\end{rem}

We refer the reader to \cite{Miy2} for the relationship between the condition (\ref{Q2}) and Donaldson's functional restricted to diagonal metrics on Higgs bundles.

\medskip
\noindent
{\bf Acknowledgements.} I wish to express my gratitude to Ryushi Goto and Hisashi Kasuya for their valuable comments and helpful advice. I am very grateful to Yoshinori Hashimoto for his valuable discussions, and for asking me about the relationship between condition (\ref{gamma}) and the definition of the usual stability condition.

\noindent
E-mail address 1: natsuo.miyatake.e8@tohoku.ac.jp

\noindent
E-mail address 2: natsuo.m.math@gmail.com \\

\noindent
Mathematical Science Center for Co-creative Society, Tohoku University, 468-1 Aramaki Azaaoba, Aoba-ku, Sendai 980-0845, Japan.


\begin{thebibliography}{99}
\bibitem{AG1}L. \'{A}lvarez-C\'{o}nsul and O. Garc\'{i}a-Prada, {\it Hitchin-Kobayashi correspondence, quivers, and vortices}, Comm. Math. Phys. 238 (2003), no. 1-2, 1–33. 
\bibitem{Bap1}J. M. Baptista, {\it Moduli spaces of Abelian vortices on K\"ahler manifolds}, arXiv:1211.0012.
\bibitem{Bar1} D. Baraglia, {\it Cyclic Higgs bundles and the affine Toda equations}, Geom. Dedicata 174 (2015), 25–42.
\bibitem{Bar2} D. Baraglia, {\it $G_2$ geometry and integrable systems}, Ph.D. Thesis, University of Oxford, 2009, arXiv:1002.1767.
\bibitem{Dem1} J.-P. Demailly, {\it Analytic methods in algebraic geometry}, Surveys of Modern
Mathematics, 1. International Press, Somerville, MA; Higher Education Press,
Beijing, 2012. viii+231 pp. ISBN: 978-1-57146-234-3.
\bibitem{DL1}S. Dai and Q. Li, {\it On cyclic Higgs bundles}, Math. Ann. 376 (2020), no. 3-4, 1225–1260.
\bibitem{GGM1} O. Garc\'{i}a-Prada, P. B. Gothen, and I. Mundet i Riera, {\it The Hitchin-Kobayashi correspondence, Higgs pairs and surface group representations}, preprint, arXiv:0909.4487, 2009.
\bibitem{Miy1} N. Miyatake, {\it Generalized Kazdan-Warner equations associated with a linear action of a torus on a complex vector space}, Geom Dedicata 214, 651–669 (2021).
\bibitem{Miy2} N. Miyatake, {\it Restriction of Donaldson's functional to diagonal metrics on Higgs bundles with non-holomorphic Higgs fields}, preprint, arXiv:2301.01485.
\bibitem{Sim1} C.T. Simpson, {\it Constructing variations of Hodge structure using Yang-Mills theory and applications to uniformization,} J. Amer. Math. Soc. 1 (1988), no. 4, 867–918. 
\end{thebibliography}
\end{document}